\documentclass[12pt]{article}

\usepackage [utf8] {inputenc}
\usepackage {t1enc}
\usepackage [english] {babel}
\usepackage {calc}
\usepackage {latexsym}
\usepackage {amssymb}
\usepackage {amsmath}
\usepackage {amsthm}
\usepackage {comment}
\usepackage {hyperref}
\usepackage {tikz-cd}
\usetikzlibrary{arrows}

\providecommand{\keywords}[1]
{
	\small	
	{\textit{Keywords---}} #1
}
\newcommand{\Ker}{\mathrm{Ker}}
\newcommand{\id}{\mathrm{id}}
\newcommand{\tr}{\operatorname{tr}}

\newtheorem{theorem}{Theorem}
\newtheorem{lemma}[theorem]{Lemma}
\newtheorem{corollary}[theorem]{Corollary}

\begin{document}
\title{The purity locus of matrix Kloosterman sums}
\author{M\'arton Erd\'elyi, Will Sawin, and \'Arp\'ad T\'oth} 

\maketitle

\begin{abstract}
	We establish the exact structure of the cohomology associated to a certain matrix exponential sum investigated in \cite{Erdelyi-Toth}. 
\end{abstract}


\section{Introduction}

Fix $\mathbb F_q$ a finite field with characteristic $p$ and algebraic closure $\overline{\mathbb{F}_q}$, $\psi$ an additive character of $\mathbb F_q$, and $n$ a natural number. Let $\mathbb{F}_{q^m}\subset\overline{\mathbb{F}_q}$ be the degree $m$ extension of $\mathbb{F}_q$ and $\psi_m=\mathrm{Tr}_{\mathbb{F}_{q^m}|\mathbb{F}_q}\circ\psi$.

For a matrix $a \in M_n(\mathbb{F}_q)=\mathbb{F}_q^{n\times n}$ the exponential sums
\[K(a,\mathbb{F}_{q^m})=K(a)=\sum_{x\in\mathrm{GL}_n(\mathbb{F}_{q^m})}\psi_m(\tr(ax+x^{-1}))\]
(here $\tr$ is the matrix trace) and the related cohomology complex
\[
H^*(a)=H^*_c ( GL_{n, \overline{\mathbb F_q}}, \mathcal L_\psi ( \tr ( a x^{-1} +x ) ))
\]
are investigated in \cite{Erdelyi-Toth}. 
The relation between the exponential sum and the complex can be made explicit by the Grothendieck trace formula (\cite{Grothendieck}): 
\[K(a,\mathbb{F}_{q^m})=\sum_{i=0}^{2d}(-1)^i\sum_{j=1}^{d_i}\iota(\lambda_j^i)^m,\]
where $d=\dim(\mathrm{GL}_n)=n^2$, $d_i=\dim(H^i(a))$, $\iota:\mathbb{Q}(\psi(1))\to\mathbb{C}$ is the embedding sending $\psi(1)$ to $\exp(2\pi i/p)$ and the $\lambda_j^i$-s are the Frobenius eigenvalues on $H^i(a)$ of some integral weight $0\leq j\leq i$ (i. e. $|\iota(\lambda_j^i)|=q^{j/2}$).

Although effective bounds were given for the sum or the weights of the Frobenius eigenvalues (not depending on the degree $i$), the exact structure of the cohomology was not determined (cf. Remark 2.):

It was proven that $H^*(a)$ is concentrated in the middle degree $i=n^2$ and all the weights are $n^2$ if $a$ is invertible and has no multiple eigenvalues in $\overline{\mathbb{F}}_q$. In the regular case (i.e. for all eigenvalues in $\overline{\mathbb{F}}_q$ the eigenspace is one-dimensional) it was pointed out that the sum admits square root cancellation, but the methods used there did not allow to decide whether in $H^*(a)$ the single nontrivial cohomology is the middle term or not (or equivalently the Frobenius eigenvalues in different degrees cancel or not).

\vspace{1em}

The sum or the complex is called pure of weight $w$ if all $\lambda_j^i$ is of weight $i+w$. The aim of this paper is to show the cohomology groups are always pure, though of weight varying depending on $a$:

\begin{theorem}\label{thmpurity}
Let $a\in M_n(\mathbb{F}_q)$ be a matrix. Let $k$ be the multiplicity of $0$ as an eigenvalue of $a$. Then $H^*(a)$ is a pure complex of weight $-k$.
\end{theorem}

Straightforward consequences of the purity are the following:
\begin{enumerate}
\item There can not be cancellation between the Frobenius eigenvalues of different degrees.
\item Knowing the value of $K(a,\mathbb{F}_{q^m})$ for all finite extension $\mathbb{F}_{q^m}|\mathbb{F}_q$ is equivalent to knowing the Frobenius eigenvalues of $H^*(a)$.
\end{enumerate}

It can also be explicitly characterized when $H^*(a)$ is concentrated in the middle degree (and thus $K(a)$ admits square root cancellation):

\begin{theorem}\label{thmreg}
The cohomology complex $H^*(a)$ is concentrated to the middle degree if and only if for each nonzero eigenvalue $\alpha$ of $a$, the eigenspace is 1-dimensional.
\end{theorem}

In the case of regular matrices one can just refer to the computations in Section 22 of \cite{Erdelyi-Toth}, but our methods put this in  a more geometric context.

For the proof of the general case an explicit description of $K(a)$ and $H^*(a)$ is given for matrices $a$ which have a unique eigenvalue $\alpha\in\mathbb{F}_q$ (compare with the recursion algorithm in Section 18 of \cite{Erdelyi-Toth}):

\begin{theorem}\label{thmJblock}
If $a$ has a unique eigenvalue $\alpha$ then let $\lambda, \bar\lambda$ be the Frobenius eigenvalues corresponding to the classical Kloosterman sum $K(\alpha)$, then
\[K(a)=(-1)^nq^{n(n-1)/2}\sum_{k=0}^n\#\left(W\leq\mathbb{F}_q^n|\dim(W)=k,aW=W\right)\lambda^k\bar\lambda^{n-k}.\]
\end{theorem}

\vspace{1em}

In order to do prove these, first a complex of sheaves is constructed on $\mathbb{A}^{n^2}$, whose fiber over the matrix $a$ is the cohomology group $H^*(a)$.

This can be constructed as \[ K = R pr_{1!} \mathcal L_\psi ( \tr ( a x^{-1} +x ))\] where $pr_1: \mathbb A^{n^2} \times GL_n \to \mathbb A^{n^2}$ is the projection.

\begin{theorem}\label{thmsheaf}
The shifted complex $K[2 n^2]$ is an irreducible pure perverse sheaf.
\end{theorem}

Let $U \subset \mathbb A^{n^2}$ be the regular semisimple locus, and let $j \colon U \to \mathbb A^{n^2}$ be the open immersion. Let $\widetilde{\mathbb A}^{n^2}$ be the space parameterizing pairs of a matrix $a \in \mathbb A^{n^2}$ and a complete flag $F$ of linear subspaces $F^i$ of $\mathbb A^n$ such that $a$ preserves each of the subspaces $F^i$. Let $\pi \colon \widetilde{\mathbb A}^{n^2}\to \mathbb A^{n^2}$ be the map forgetting the flag $F$ (Grothendieck's simultaneous resolution).

Let $\widetilde{U}$ be the inverse image of $U$ under $\pi$, let $u \colon \widetilde{U} \to \widetilde{\mathbb A}^{n^2}$ be the open immersion, and let $\rho \colon \widetilde{U} \to U$ be the projection. Note that $\widetilde{U}$ parameterizes pairs $(a, F)$ where $a$ preserves $F$ and $a$ is regular semisimple. In this case, $a$ has $n$ one-dimensional eigenspaces for its $n$ distinct eigenvalues, and $F^i$ is invariant if and only if it is a sum of these eigenspaces, so choices of a complete flag are in one-to-one bijections with orderings of the eigenvalues.

Finally, let $\lambda_i \colon \widetilde{\mathbb A}^{n^2}\to \mathbb A^1$ send a pair $(a,F)$ to the unique eigenvalue of $a$ acting on the one-dimensional space $F^i / F^{i-1}$. Thus, $\lambda_1 (a,F) \dots, \lambda_n(a,F)$ are the eigenvalues of $a$ with multiplicity.

By Grothendieck's simultaneous resolution the argument of \cite{Erdelyi-Toth} for regular semisimple elements can be translated to the fact that $K$ is a summand of a complex of sheaves related to a tensor product of hyper-Kloosterman sheaves:

\begin{theorem}\label{thmsummand}
$K$ is a summand of \[R\pi_*   \bigotimes_{i=1}^n \lambda_i^* \mathcal{K}\ell_2   [-n^2] (n (n-1)/2),\]
where $\mathcal{K}\ell_2$ is the classical hyper-Kloosterman sheaf on $\mathbb A^1$ defined by Katz (\cite{Katz-monodromy}).
\end{theorem}

Then Theorem \ref{thmpurity} is a direct consequence of  Springer's result on the cohomology of the Springer fiber \cite{Springer2}.

When $a$ is regular, then $\pi^{-1}(a)$ consists only points, which allows us to prove Theorem \ref{thmreg} in this case.

\vspace{1em}

One can motivate Theorem \ref{thmJblock} by relating the geometry of $K$ to the Springer correspondence:

If the classical Kloosterman sum $K(\alpha)$ has Frobenius eigenvalues $\lambda$ and $\bar\lambda$, then an $n$-fold tensor product of Kloosterman sheaves splits into one-dimensional summands with eigenvalues $\lambda^k$ and $\bar\lambda^{n-k}$. The action of the Weyl group $S_n$ permutes the terms and $S_k\times S_{n-k}\leq S_n$ shows up as the stabilizer of an eigenvector. This suggests that the invariants of $ S_k\times S_{n-k}$ on the cohomology are given by the cohomology of the space of $k$-dimensional subspaces fixed by $a$, and that leads to the formula in Theorem \ref{thmJblock}.

This is proven using elementary arguments and a recursion formula for the sum \(K(a)\), (\cite{Erdelyi-Toth}, Theorem 18.1.), which also justifies the above geometric description, as it is the only one that matches the sum in Theorem \ref{thmJblock}.

The computation of some weights allows to conclude Theorem \ref{thmreg} for non-regular matrices.

\vspace{1em}

One may be naturally inclined to speculate how the results generalize to reductive subgroups of \(GL_n\). We will only comment on the question of purity, for which understanding the generic situation as in Theorem~\ref{thmsummand} is crucial. This result seems to be closely tied up with the underlying standard representation of the group \(GL_n\). Other representations will definitely lead to more elaborate structures. To illustrate  consider the subgroup \(H=Sym^2(GL_2)\) of \(GL_3\) and  \(K_H(a)=\sum_{x\in H}\psi(ah+h^{-1}) \). 
Then a direct calculation shows that the associated sheaf cohomology is not pure, even for \(a\in H\) regular semisimple.

%
%

\vspace{1em}

A brief description of the paper follows. Our argument naturally splits to two parts:

In Section \ref{sgeom} the Theorems \ref{thmsheaf}, \ref{thmsummand} and \ref{thmpurity} are proven by geometric arguments.

In Section \ref{selm} Theorem \ref{thmJblock} is proven by combinatorial calculations and then the remaining part of Theorem \ref{thmreg}.

{\em Acknowledgments.} 

Erd\'elyi was supported by NKFIH Research Grants FK-127906 and K-135885.

Sawin was supported by NSF grant DMS-2101491.

Toth was supported by by the R\'enyi Institute Lend\"ulet Automorphic Research Group, and by NKFIH Research Grants K 135885.

\section{Geometric arguments}\label{sgeom}

\begin{proof}[Proof of Theorem \ref{thmsheaf}] Consider first the space $\mathbb A^{n^2} \times \mathbb A^{n^2}$ parameterizing pairs of $n \times n$ matrices $(x,y)$. On this space, we have the closed subset $Z$ given by the equation $xy=\id$, which is isomorphic to $GL_n$ under the isomorphism sending $(x, x^{-1})$ to $x$.  Let $i \colon GL_n  \to \mathbb A^{n^2} \times \mathbb A^{n^2}$ be the closed immersion. Then since $GL_n$ is smooth and irreducible of dimension $n^2$, $\mathbb Q_\ell[n^2]$ is an irreducible perverse sheaf on $GL_n$. Then since $i$ is a closed immersion, $i_* \mathbb Q_\ell[n^2]$ is an irreducible perverse sheaf on $\mathbb A^{n^2} \times \mathbb A^{n^2}$.

The Fourier transform is known to preserve irreducible perverse sheaves (perversity from \cite[Corollary 2.1.5(ii)]{Katz-Laumon} and irreducibility follows trivially from \cite[III, Theorem 8.1(3)]{Kiehl-Weissauer}). Thus, the Fourier transform of $i_* \mathbb Q_\ell[n^2]$ is an irreducible perverse sheaf. Next, let us calculate this Fourier transform.

By definition \cite[Definition 2.1.1]{Katz-Laumon}, the Fourier transform of a complex $L$ on $\mathbb A^{N}$ is given by  $ R pr_{l!} ( pr_r^* L \otimes \mathcal L_\psi ( \mu))[N]$ where $pr_l, pr_r$ are the two projections $\mathbb A^N \times \mathbb A^N  \to \mathbb A^N$ and $\mu$ is a nondegenerate bilinear form on $\mathbb A^N \times \mathbb A^N $. We can specialize this formula for $N = 2n^2$, viewing each copy of $\mathbb A^N$ as $\mathbb A^{n^2} \times \mathbb A^{n^2}$. A suitable bilinear takes two pairs of matrices $((a,b),(x,y))$ to $\tr( a y + bx)$.  Having done this, the Fourier transform of $i_* \mathbb Q_\ell[n^2]$ is
\[  R  pr_{12!} ( pr^{34*} i_* \mathbb Q_\ell [ n^2] \otimes \mathcal L_\psi (  \tr( a y + bx) )) [ 2n^2]\]
Here $pr_{12}$ and $pr_{34}$ are the maps from $(\mathbb A^{n^2})^4 $ to $(\mathbb A^{n^2})^2$ given by projection onto the first two and last two factors respectively.

Since $i$ is a closed immersion, $i_* = i_!$. By proper base change \cite[XVII, Proposition 5.2.8]{sga4-3}, $pr^{34*} i_! \mathbb Q_\ell = ( (id) \times i)_! \mathbb Q_\ell$ where $(id) \times i \colon  \mathbb A^{n^2} \times \mathbb A^{n^2} \times GL_n \to (\mathbb A^{n^2})^4$ sends $(a,b,x)$ to $(a,b,x,x^{-1})$.  By the projection formula \cite[XVII, Proposition 5.2.9]{sga4-3},
\begin{align*}
( (id) \times i)_! \mathbb Q_\ell \otimes \mathcal L_\psi (  \tr( a y + bx) )) = ((id)\times i)_!  (id \times i)^* \mathcal L_\psi (  \tr( a y + bx) )) =\\
((id)\times i)_!   \mathcal L_\psi ( \tr ( ax^{-1} + bx))
\end{align*}

Thus, the Fourier transform of $i_* \mathbb Q_\ell[n^2]$ is 
\[  R  pr_{12!} ((id)\times i)_!   \mathcal L_\psi ( \tr ( ax^{-1} + bx)) [3n^2]\]
which by functoriality of compactly supported pushforward \cite[XVII, Theorem 5.1.8(a)]{sga4-3} is
\[R ( pr_{12} \circ (id)\times i))_! \mathcal L_\psi ( \tr ( ax^{-1} + bx)) [3n^2].\]

The composition $R ( pr_{12} \circ (id)\times i))_! $ may also be called $pr_{12}$ as it represents projection onto the first two factors $a,b$. So $R pr_{12!} \mathcal L_\psi( \tr( ax^{-1} + bx))[3n^2]$ is an irreducible perverse sheaf.

This irreducible sheaf remains irreducible when restricted to the open locus $\mathbb A^{n^2} \times GL_n$ where $b$ is invertible. (The only other possibility is that it becomes zero on this subset, which is easy to rule out by the calculations already in \cite{Erdelyi-Toth}.) On that locus, performing a change of variables substituting $b^{-1} x$ for $x$, we get
\[R pr_{12!} \mathcal L_\psi( \tr( ax^{-1}b + x))[3n^2] =R pr_{12!} \mathcal L_\psi( \tr( bax^{-1} + x))[3n^2]  \] which by proper base change \cite[XVII, Proposition 5.2.8]{sga4-3} is $m^* K[n^2]$ where  $m \colon  (a,b) \mapsto ba$ from $\mathbb A^{n^2} \times GL_n $ to  $GL_n$ is the multiplication map. Because $m$ is smooth of relative dimension $n^2$ with nonempty, geometrically connected fibers, the function $ L \mapsto m^* L [n^2]$ preserves the perverse $t$-structure and moreover is a fully faithful functor on perverse sheaves \cite[Proposition 4.2.5]{bbd}.

It follows that $K$ is an irreducible perverse sheaf: Assuming for contradiction that ${}^p\mathcal H^i(K)\neq 0$ for some $i\neq 0$ then since $m$ preserves the perverse $t$-structure and is fully faithful we have ${}^p \mathcal H^i( m^* K[n^2]) = m^* {}^p\mathcal H^i(K)[n^2]  \neq 0$, a contradiction, so  ${}^p\mathcal H^i(K)=0$ for all $i\neq 0$, i.e. $K$ is perverse, and assuming for contradiction that $K$ is reducible, so there is a proper nontrivial subobject $L$ of $K$, then $m^* L [n^2]$ would be a proper nontrivial subobject of $m^* K[n^2]$, another contradiction, thus $K$ is irreducible.

Since $K$ is mixed by \cite[Variant 6.2.3]{weil-ii}, it is pure by \cite[Corollary 5.3.4]{bbd}.
\end{proof}

We state some general facts about perverse sheaves that are relatively standard but may not appear in the literature in exactly the form we need them.

\begin{lemma}\label{trace-equals-criterion} Let $K_1$ and $K_2$ be two perverse sheaves on a variety $X$ over a finite field $\mathbb F_q$. Assume that for each extension $\mathbb F_{q^e}$ of $\mathbb F_q$, for each $x \in X( \mathbb F_{q^e})$, we have the equality of trace functions
\begin{equation}\label{trace-functions-equal} \sum_i (-1)^i \tr (\operatorname{Frob}_{q^e}, \mathcal H^i ( K_1)_{x}) =\sum_i (-1)^i \tr (\operatorname{Frob}_{q^e}, \mathcal H^i ( K_2)_{x})\end{equation}
and that $K_2$ is irreducible. Then $K_1 \cong K_2$. \end{lemma}

\begin{proof}We first check that for bounded complexes $K_1$, $K_2$ of constructible $\ell$-adic sheaves satisfying the equality of trace functions \eqref{trace-functions-equal}, classes $\sum_i (-1)^i [\mathcal H^i(K_1)]$ and $\sum_i (-1)^i [\mathcal H^i(K_2)]$ in the Grothendieck group of constructible $\ell$-adic sheaves of $K_1$ and $K_2$ are equal. Since both the assumption and conclusion involve an alternating sum, we reduce immediately to the case that $K_1$ and $K_2$ are constructible $\ell$-adic sheaves. There exists a stratification of $X$ such that both $K_1$ and $K_2$ are lisse on each stratum, and since we can write each as the sum in the Grothendieck group of its restriction to each stratum, it suffices to handle the case where $K_1$ and $K_2$ are lisse. In this case, we are working in the Grothendieck group of representations of the fundamental group. By the Chebotarev density theorem, the equality of traces of Frobenius implies the equality of traces of every element of $\pi_1$. The result then follows from the fact, in character theory, that two representations with the same character are equal in the Grothendieck group of representations.

The class of a complex of perverse sheaves in the Grothendieck group is the alternating sum of the classes of its perverse homology sheaves. From this, one can see that $K_1$ and $K_2$ agree in the Grothendieck group of perverse sheaves. But, since the category of perverse sheaves is Artinian, the Grothendieck group of perverse sheaves is the free group on the isomorphism classes of irreducible perverse sheaves. It follows that, in the Jordan-H\"older decomposition of $K_1$ in the category of perverse sheaves, the only irreducible component is $K_2$, and it occurs with multiplicity $1$, i.e. $K_1$ is isomorphic to $K_2$.

\end{proof}

\begin{lemma}\label{direct-summand-criterion} Let $X$ be a variety over a finite field, $U$ an open set of $X$, and $j \colon U \to X$ an open immersion. Let $K_1$ be an irreducible pure perverse sheaf on $X$ and $K_2$ a pure complex on $X$.

If $j^* K_1$ is a summand of $j^* K_2$ then $K_1$ is a summand of $K_2$ unless $j^* K_1=0$.
 
 \end{lemma}
 
 \begin{proof} Since perversity, by definition, is preserved by restriction to an open set, we see that $j^* K_1$ is perverse and thus is a summand of ${}^p \mathcal H^0 ( j^* K_2) = j^* {}^p \mathcal H^0(K_2)$. Since ${}^p \mathcal H^0(K_2)$ is a summand of $K_2$ by \cite[Theorem 5.4.5]{bbd}, it suffices to prove that $K_1$ is a summand of ${}^p \mathcal H^0(K_2)$, i.e. we can reduce to the case where $K_2$ is perverse.
 
 By \cite[Corollary 5.3.11]{bbd}, for $i$ the closed immersion of the closed complement of $U$, we can write $K_2$ as a sum $j_{!*} K_2' \oplus  i_* K_2''$. We have
 \[ j^* K_2 = j^* (j_{!*} K_2' \oplus  i_* K_2'') =j^* j_{!*} K_2' \oplus  j^* i_* K_2''= K_2' \oplus 0 \]
 so $j^* K_1$ is a summand of $ K_2'$ and thus $j_{!*} j^* K_1$ is a summand of $j_{!*} K_2'$ which is a summand of $K_2$. So it suffices to show $j_{!*} j^* K_1\cong K_1$.
 
 But now applying \cite[Corollary 5.3.11]{bbd} to $K_1$, we have $K_1 = j_{!*} K_1' \oplus i_* K_1''$. Since $K_1$ is irreducible, only one of these terms can be nonzero, and since $j^* i_* K_1''=0$, if $j^* K_1\neq 0$ then we must have $K_1''=0$, i.e. $K_1 =j_{!*} K_1'$. But then 
 \[j_{!*} j^* K_1 = j_{!*} (j^*  j_{!* } K_1' ) = j_{!*} K_1' = K_1 ,\] as desired.
 
 \end{proof}

\begin{lemma}\label{etp-irreducible} Let $X_1$ and $X_2$ be two varieties over a field $\kappa$, and let $K_1$ and $K_2$ be geometrically irreducible perverse sheaves on $X_1$ and $X_2$. Then the external tensor product $K_1 \boxtimes K_2$ is a geometrically irreducible perverse sheaf. \end{lemma}

\begin{proof} By \cite[Proposition 4.2.8]{bbd}, $K_1\boxtimes K_2$ is perverse.

By passing to the algebraic closure, we may assume $\kappa$ is algebraically closed, so geometrically irreducible is the same as irreducible.

By replacing $X_1$ and $X_2$ by the supports of $K_1$ and $K_2$, we may assume $K_1$ is supported on $X_1$ and $K_2$ is supported on $X_2$. Then $K_1$ and $K_2$ are middle extensions of irreducible local systems on smooth affine open subsets $U_1$ and $U_2$ of $K_1$ and $K_2$ \cite[Theorem 4.3.2(ii)]{bbd}. Restricted to the open set $U_1 \times U_2$, the perverse sheaf $K_1\boxtimes K_2$ is the external tensor product of irreducible local systems and thus is irreducible.

Thus if $K_1 \boxtimes K_2$ is not reducible, say it is the extension of a nontrivial quotient $Q$ by a nontrivial subobject $S$, at least one of $Q$ or $S$ must vanish on restriction to $U$ and thus have support contained is the complement $Z$ of $U_1 \times U_2 $. Since perverse sheaves and external tensor products are stable under duality, by dualizing everything we may assume without loss of generality that $Q$ has support contained in $Z$, and thus can be written as $i_* Q'$ for a perverse sheaf $Q'$ on $Z$. We have a nontrivial morphism $K_1 \boxtimes K_2 \to Q= i_* Q'$, thus by adjunction a nontrivial morphism $i^* (K_1 \boxtimes K_2) \to Q'$.

However, we can check that $i^* (K_1 \boxtimes K_2) [-1]$ is semiperverse. The complement of $U$ is a union $(X_1 \setminus U_1) \times X_2 \cup X \times (X_2 \setminus U_2)$, and it suffices to prove this on $(X_1 \setminus U_1) \times X_2 $ and $ X \times (X_2 \setminus U_2)$ separately, or, without loss of generality $(X_1 \setminus U_1) $. But the restriction of $K_1[-1]$ to $(X_1 \setminus U_1) $ is perverse by \cite[Corollary 4.1.12]{bbd}, so its external tensor product with $K_2$ is (semi)perverse by \cite[Proposition 4.2.8]{bbd}.

Because $i^* (K_1 \boxtimes K_2) [-1]$ is semiperverse, it can have no nontrivial morphism to the perverse sheaf $Q'$ by the definition of $t$-structure. \end{proof}

\begin{proof}[Proof of Theorem \ref{thmsummand}]
We begin by proving the existence of an isomorphism
\begin{equation}\label{key-isomorphism} \rho^* j^* K[2n^2] = u^*  \bigotimes_{i=1}^n \lambda_i^* \mathcal{K}\ell_2   [n^2] (-n (n-1)/2).\end{equation}

First we show how the formula of \cite[Theorem 1.1(2)]{Erdelyi-Toth} implies the traces of Frobenius on the stalks of the two sides of \eqref{key-isomorphism} at any fixed point of $\widetilde{U}(\mathbb F_q)$ are equal.
 A point of $\widetilde{U}(\mathbb F_q)$ is a matrix $a \in M_n(\mathbb F_q)$ with distinct eigenvalues, together with an $a$-stable flag $F$ on $\mathbb F_q^n$. The $i$th eigenvalue $\lambda_i$ is the eigenvalue of the action of $a$ on $F^i / F^{i-1}$, which, being the eigenvalue of a $1\times 1$ matrix over $\mathbb F_q$, lies in $\mathbb F_q$. So $a$ is a matrix with $n$ distinct eigenvalus, all in $\mathbb F_q$.
 
 The trace of Frobenius on the stalk of $\rho^* j^* K[2n^2]$ at $(a,F)$ is the trace of Frobenius on the stalk of $K[2n^2]$ at $j(\rho((a,F)))=a$, which by the Lefschetz fixed point formula is the matrix Kloosterman sum $K(a)$.
 
 On the other hand, the trace of Frobenius on the stalk of \[u^*  \bigotimes_{i=1}^n \lambda_i^* \mathcal{K}\ell_2   [n^2] (-n (n-1)/2)\] at $(a,F)$ is  $(-1)^{n^2} q^{ n (n-1)/2}$ times the trace of Frobenius on $u^*  \bigotimes_{i=1}^n \lambda_i^* \mathcal{K}\ell_2$, which itself is the product for $i$ from $1$ to $n$ of the trace of Frobenius on the stalk of $\mathcal{K}\ell_2$ at $\lambda_i ( (a,F))$. The stalk of $\mathcal{K}\ell_2$ at a point $\lambda_i$ is, by definition and the Lefschetz fixed point formula, $(-1)^n$ times the Kloosterman sum $K(\lambda_i)$, so the equality of traces follows from \cite[Theorem 1.1(2)]{Erdelyi-Toth}
\[K(a)=q^{n(n-1)/2}\prod_{j=1}^nK(\lambda_i)\]
once we realize that $(-1)^{n^2+n}=1$. Note that this works for an arbitrary finite field $\mathbb F_q$.

It follows from Theorem \ref{thmsheaf} and the preservation of perversity under \'{e}tale pullbacks \cite[first line of p. 109, $d=0$ case]{bbd} that $\rho^* j^* K[2n^2] $ is perverse. If we let $(\prod_i \lambda_i)$ be the morphism $\mathbb A^{n^2} \to \mathbb A^n$ whose $i$th coordinate is $\lambda_i$, then
\[  u^*  \bigotimes_{i=1}^n \lambda_i^* \mathcal{K}\ell_2   [n^2] (-n (n-1)/2)=  u^*   \bigl(\prod_ i \lambda_i\bigr)^* \boxtimes_{i=1}^n  \mathcal{K}\ell_2   [n^2] (-n (n-1)/2)\]
and since $\mathcal {K}\ell_2[1]$ is perverse and geometrically irreducible since it arises by \cite[Corollary 4.1.2(i)]{Katz-monodromy} from the construction of \cite[5.2.2(a)]{bbd}, $\boxtimes_{i=1}^n  \mathcal{K}\ell_2   [n] $ is perverse and geometrically irreducible by Lemma \ref{etp-irreducible}. The map $(\prod_i \lambda_i)^* $ is smooth of relative dimension $n^2-n$, and $u$ is an open immersion, so  $u^*   (\prod_i \lambda_i)^* \boxtimes_{i=1}^n  \mathcal{K}\ell_2   [n^2]$ is perverse and geometrically irreducible by \cite[Proposition 4.2.5]{bbd}. Tate twisting does not affect perversity and geometrically irreducibility, since it does not change the complex over an algebraically closed field at all.

By Lemma \ref{trace-equals-criterion}, this implies \eqref{key-isomorphism}. From \eqref{key-isomorphism} we obtain \[ \rho_* \rho^* j^* K[2n^2] = \rho_* u^*  \bigotimes_{i=1}^n \lambda_i^* \mathcal{K}\ell_2   [n^2] (-n (n-1)/2).\]

Since $\rho$ is a finite \'{e}tale morphism, by \cite[IX, (5.1.4)]{sga4-3} there are morphisms $j^* K[2n^2] \to \rho_* \rho^* j^*K [2n^2] \to j^* K[2n^2]$ whose composition is multiplication by the degree of $\rho$ and thus is invertible, meaning that  $j^* K [2n^2]$ is a summand of $\rho_* \rho^* j^* K[2n^2]$. It follows that $j^* K [2n^2]$ is a summand of \[\rho_*  u^*  \bigotimes_{i=1}^n \lambda_i^* \mathcal{K}\ell_2   [n^2] (n (n-1)/2),\] which, since $\pi$ is proper, is isomorphic by proper base change \cite[XVII, Proposition 5.2.8]{sga4-3} to \begin{equation}\label{complex-its-summand-of} j^* R\pi_*   \bigotimes_{i=1}^n \lambda_i^* \mathcal{K}\ell_2   [n^2] (n (n-1)/2).\end{equation}

Again using the isomorphism $ \bigotimes_{i=1}^n \lambda_i^* \mathcal{K}\ell_2  =     (\prod_i \lambda_i)^* \boxtimes_{i=1}^n  \mathcal{K}\ell_2 $, since $ \mathcal{K}\ell_2 $ is pure of weight $1$, $ \boxtimes_{i=1}^n  \mathcal{K}\ell_2 $ is pure of weight $n$ \cite[(5.1.14.1) and (5.1.14.1*)]{bbd}, and because $     (\prod_i \lambda_i)$ is smooth, $\bigotimes_{i=1}^n \lambda_i^* \mathcal{K}\ell_2$ is pure of weight $n$ \cite[Stabilities 5.1.14(i,i*)]{bbd}.   Because $\pi$ is proper, $R \pi_*$ preserves weights, as does $j^*$ \cite[Stabilities 5.1.14(i,i*)]{bbd}. The shift $[n^2]$ raises weights by $n^2$, and the Tate twist $(- n(n-1)/2)$ raises them by $n(n-1)$, so altogether \eqref{complex-its-summand-of} is pure of weight $2n^2$.

We apply Lemma \ref{direct-summand-criterion} with $K_1= K[2n^2]$ and $K_2 = R\pi_*   \bigotimes_{i=1}^n \lambda_i^* \mathcal{K}\ell_2   [n^2] (n (n-1)/2)$. We have checked all the conditions except that $j^* K[2n^2]$ is nonzero, but $j_* K[2n^2]=0$ is easy to rule out here since that would imply its trace function is identically zero, but the formula of \cite[Theorem 1.1.2(2)]{Erdelyi-Toth} is manifestly not identically zero (noting that the standard Kloosterman sum is nonzero since it is congruent to $p-1 \equiv -1\not\equiv 0$ modulo $e^{ 2\pi i /p}-1$), so we conclude that $K[2n^2]$ is a summand of \[R\pi_*   \bigotimes_{i=1}^n \lambda_i^* \mathcal{K}\ell_2   [n^2] (n (n-1)/2)\] and shifting both sides by $[-2n^2]$ we get the desired statement.
\end{proof}

We specialize this fact to a particular value of $a$.

\begin{corollary}\label{corsummand} For any matrix $a$ over $\mathbb F_q$, $H^j(a)$ is a summand of \begin{equation}\label{sum-of-geometrically} H^{j-n^2} \Bigl( \pi^{-1}(a), \bigotimes_{i=1}^n \lambda_i^* \mathcal{K}\ell_2   (n (n-1)/2)\Bigr). \end{equation}  \end{corollary}

\begin{proof} This follows immediately from Theorem \ref{thmsummand} after taking stalks at $a$ and applying proper base change \cite[XVII, Proposition 5.2.8]{sga4-3}.\end{proof} 

This enables us to give geometric proofs of Theorems \ref{thmpurity} and \ref{thmreg}. 
\begin{proof}[Proof of Theorem \ref{thmreg} for $a$ regular]
 If $a$ is regular then $\pi^{-1}(a)$ consists of finitely many points, as $a$ has finitely many invariant subspaces of each dimension. Thus \eqref{sum-of-geometrically} vanishes unless $j = n^2$, and so $H^j(a)$ vanishes unless $j=n^2$ by Corollary \ref{corsummand}.
\end{proof}

\begin{proof}[Proof of Theorem \ref{thmpurity}] We can write $\pi^{-1}(a)$ as a disjoint union of connected components, say $Z_1,\dots, Z_n$. Then we have 
\[ H^{j-n^2} \Bigl( \pi^{-1}(a), \bigotimes_{i=1}^n \lambda_i^* \mathcal{K}\ell_2   (n (n-1)/2)\Bigr) = \bigoplus_{r=1}^n  H^{j-n^2} \Bigl( Z_r,  \bigotimes_{i=1}^n \lambda_i^* \mathcal{K}\ell_2   (n (n-1)/2)\Bigr).\]
Now the maps $\lambda_i$ can take only finitely many values on $\pi^{-1}(a)$, those being the eigenvalues of $a$. It follows that $\lambda_i$ is constant on each connected component $Z_r$. (The image of a connected component under $\lambda_i$ must be a connected component of the image of $\lambda_i$, but if the image is finite, then each connected component is a point.)  Letting $\lambda_{i} (Z_r)$ be this constant value, then $\lambda_i^* \mathcal {K}\ell_2$ is the tensor product of the constant sheaf with the Galois representation $(\mathcal{K}\ell_{2} )_{\lambda_i(Z_r)}$. This gives
\[  H^{j-n^2} \Bigl( Z_r,  \bigotimes_{i=1}^n \lambda_i^* \mathcal{K}\ell_2   (n (n-1)/2)\Bigr) =  H^{j-n^2} \Bigl( Z_r,  \mathbb Q_\ell \otimes  \bigotimes_{i=1}^n (\mathcal{K}\ell_{2} )_{\lambda_i(Z_r)} (n (n-1)/2)\Bigr) \] \[= H^{j-n^2} \Bigl( Z_r,  \mathbb Q_\ell \Bigr) \otimes  \bigotimes_{i=1}^n (\mathcal{K}\ell_{2} )_{\lambda_i(Z_r)} (n (n-1)/2). \]

By \cite[Theorem 1]{Springer2}, $H^{j-n^2} \Bigl( Z_r,  \mathbb Q_\ell \Bigr) $ is pure of weight $j-n^2$.  Using the fact that the stalk of $ \mathcal{K}\ell_2 $ is pure of weight $1$ except at $0$ where it is pure of weight $0$ \cite[Theorem 4.1.1(1) and Theorem 7.4.3]{Katz-monodromy}, and the fact that the number of $i$ with $\lambda_i=0$ is the multiplicity of $0$ as an eigenvalue and thus is equal to $k$, we see that  $  \bigotimes_{i=1}^n (\mathcal{K}\ell_{2} )_{\lambda_i(Z_r)} $ is pure of weight $n-k$.  Finally  $\mathbb Q_\ell(n (n-1)/2)$ is pure of weight $n(n-1)$.

So their tensor product is pure of weight $(n^2-j )+ (n-k) + n(n-1) = j-k$. Thus the same thing is true for $H_j(a)$, because $H_j(a)$ is a summand of a sum of these tensor products.\end{proof}

\section{Combinatorial arguments}\label{selm}

Now we turn to the proof of Theorem \ref{thmJblock}. Therefore fix $a\in M_n(\mathbb{F}_q)$ with a unique nonzero eigenvalue $\alpha$ and let
\[r=\min(j|(a-\alpha\id)^j=0).\]

There is a filtration on \(V=\mathbb{F}_q^n\) given by
\[
V_j=\mathrm{Ker}((a-\alpha\id)^j)\leq\mathbb{F}_q^n,
\]
wich we extend by restriction ot any subspace $W\leq\mathbb{F}^n$, that is we let $W_j=\mathrm{Ker}((a-\alpha\id)^j|_W)=W\cap V_j$.

The associated numerical data  \(\mu_j=\dim(V_j/V_{j-1})\) defines a partition \(\mu\) of \(n=\dim V\), 
\begin{equation}\label{eq:mu-def}
\mu=[\mu_1\geq\mu_2\geq\dots\geq\mu_r]\vdash n.
\end{equation}  
The condition \(\sum_{j=1}^r \mu_j=n\) is obvious, and for $0\leq j<r-1$ the map \(a-\alpha\id:V/V_{j+1}\to V/V_j\) is injective and maps \(V_{j+2}/V_{j+1}\) into \(V_{j+1}/V_j\) which shows that \(\mu_1\geq\mu_2\geq\dots\geq\mu_r\). Note that the dual partition \(\mu'\) codifies the Jordan block structure of \(a\), see e.g. \cite{Springer-Steinberg} Chapter 4.

The following counting result and its recursive version will be needed for the proof of Theorem~\ref{thmJblock}. They involve a decomposition of the numbers
\[\#\left(W\leq\mathbb{F}_q^n|\dim(W)=k,aW=W\right)\] 
in that theorem by their filtration type. In what follows \(  {n \choose k}_{q} =
\frac {(1-q^{n})(1-q^{n-1})\cdots  (1-q^{n-k+1})} {(1-q)(1-q^{2})\cdots (1-q^{k})}\) is the $q$-binomial coefficient. 

\begin{theorem}
Let $\nu=[\nu_1\geq\nu_2\geq\dots\geq\nu_r]\vdash k$ be a partition, (with some $\nu_j$-s possibly 0) and set $\nu_{r+1}=0$.
Then
\begin{eqnarray*}
V(\mu,\nu):=\#\left(W\leq\mathbb{F}_q^n|aW\leq W, \dim(W_{j+1}/W_j)=\nu_j\mathrm{~for~}1\leq j\leq r\right)\\
=\prod_{j=1}^r\binom{\mu_j-\nu_{j+1}}{\nu_j-\nu_{j+1}}_q\cdot q^{\nu_{j+1}(\mu_j-\nu_j)}.
\end{eqnarray*}
\end{theorem}

\begin{proof}
For $0\leq j<r$ the map $\iota_j:W/W_j\to V/V_j$, $w+W_j\mapsto w+V_j$ is also injective thus (denoting $a-\alpha\id$ and $p$ the restrictions to $W$ also) we have the following diagram (with non-exact rows):

\begin{center}
\begin{tikzcd}
V/V_{j+1} \arrow[r, hook, "a-\alpha\id"] & V/V_j\arrow[r, twoheadrightarrow,"\pi"] & V/V_{j+1}\\
W/W_{j+1} \arrow[r, hook, "a-\alpha\id"] \arrow[u, hook, "\iota_{j+1}"] & W/W_j \arrow[r, twoheadrightarrow,"\pi"] \arrow[u, hook, "\iota_j"] & W/W_{j+1} \arrow[u, hook, "\iota_{j+1}"]
\end{tikzcd} 
\end{center}

Starting from $j=r-1$ and decreasing $j$ we can enumerate all possibilities for $W/W_j\leq V/V_j$.

In the case $j=r-1$ it is needed to choose the image of $\iota_r$ in $V_r/V_{r-1}$, i. e. a $\nu_r$-dimensional subspace in a $\mu_r$-dimensional space. There are clearly $\displaystyle{\binom{\mu_r}{\nu_r}_q=\binom{\mu_r-\nu_{r+1}}{\nu_r-\nu_{r+1}}_q\cdot q^{\nu_{r+1}(\mu_r-\nu_r)}}$ possibilities by the definition of $q$-binomial coefficients.

For $j<r-1$ fix $W/W_{j+1}\leq V/V_{j+1}$ from the previous step. We count the number of possibilities for $W/W_j\leq V/V_j$.

The conditions on $W/W_j$ can be reformulated as follows:

\begin{enumerate}
\item $\pi(W/W_j)=W/W_{j+1}$,
\item $(a-\alpha\id)(W/W_{j+1})\leq W/W_j$ (as $aW=W$) and
\item $\dim \Ker(\pi|_{W/W_j})=\dim (W_j/W_{j-1})=\nu_j$.
\end{enumerate}

First look at $\Ker(\pi|_{W/W_j})=W_{j+1}/W_j$. We have
\[(a-\alpha\id)(W_{j+2}/W_{j+1})\leq W_{j+1}/W_j.\]
Thus $W_{j+1}/W_j$ is a $\nu_j$-dimensional subspace of $V_{j+1}/V_j$ containing a fixed $\nu_{j+1}$-dimensional subspace. There are $\displaystyle{\binom{\mu_j-\nu_{j+1}}{\nu_j-\nu_{j+1}}_q}$ possibilities for it.

Now we have
\[W'/W_j:=(a-\alpha\id)(W/W_{j+1})+W_{j+1}/W_j\leq W/W_j\]
of codimension $\nu_{j+1}$. Counting dimensions we get
\[W''/W_{j+1}:=\pi(W'/W_j)\leq W/W_{j+1}\]
is also of codimension $\nu_{j+1}$ as $\dim \Ker(\pi|_{W/W_j})=\nu_j$.

Fix a basis of $W''/W_{j+1}$ and extend it to a basis of $W/W_{j+1}$ with some $w_1+W_{j+1},w_2+W_{j+1},\dots,w_{\nu_{j+1}}+W_{j+1}$. Then $W/W_j$ contains exactly one of the preimages $\pi|_{W/W_j}^{-1}(w_k+W_{j+1})$ modulo $\Ker(\pi|_{W/W_j})$, so there are $q^{\mu_j-\nu_j}$ choices for each.

All in all there are exactly $\displaystyle{\binom{\mu_j-\nu_{j+1}}{\nu_j-\nu_{j+1}}_q\cdot q^{\nu_{j+1}(\mu_j-\nu_j)}}$ ways to extend $W/W_{j+1}$ to $W/W_j$.

Finally for $j=0$ we have $W/W_j=W\leq V$, so the number of the subspaces $W$ is as in the statement.
\end{proof}

We will now give a recursive formula for the numbers \(V(\mu,\nu)\). Therefore for a partition \(\mu\) we let 
\begin{align}\label{eq:mu-primes}
\mu'=&[\mu_1\geq\mu_2\geq\dots\geq\mu_{r-1}\geq\mu_r-1],
\nonumber \\ \mu''=&[\mu_1\geq\mu_2\geq\dots\geq\mu_{r-1}-1\geq\mu_r-1]\\ \mu'''=&[\mu_1\geq\mu_2\geq\dots\geq\mu_{r-1}\geq\mu_r-2]
\nonumber
\end{align}
whenever these operations lead to partitions. If any of these operations results in a collection of numbers with a negative entry (i. e. $\mu''$ if $r=1$ or $\mu'''$ if $\mu_r=1$) then the operation is undefined (or interpreted as \(\emptyset\), not to be confused with the unique partition \([]\) of \(0\)).
With this notation we have
\begin{lemma}\label{q-binom-eqn}
	\begin{equation*}
		V(\mu,\nu)=V(\mu',\nu)+V(\mu',\nu')-q^{\mu_r-1}V(\mu'',\nu')+(q^{\mu_r-1}-1)V(\mu''',\nu')
	\end{equation*}
with the interpretation in (\ref{eq:mu-primes}), esp. that terms with negative entries are omitted from the sum.
\end{lemma}

\begin{proof}
The proof is a technical, but straightforward calculation.	Set
	\[P=\prod_{j=1}^{r-2}\binom{\mu_j-\nu_{j+1}}{\nu_j-\nu_{j+1}}_q\cdot q^{\nu_{j+1}(\mu_j-\nu_j)}\]
	This appears in the explicit formula for all the terms $V(\cdot,\cdot)$ above.
	
Then
\begin{align*}
V(\mu',\nu')+(q^{\mu_r-1}-1)V(\mu''',\nu')=\\
		\binom{\mu_{r-1}-\nu_r+1}{\nu_{r-1}-\nu_r+1}_q\cdot q^{(\nu_r-1)(\mu_{r-1}-\nu_{r-1})}\left(\binom{\mu_r-1}{\nu_r-1}_q+(q^{\mu_r-1}-1)\binom{\mu_r-2}{\nu_r-1}_q\right)P=\\
		=q^{\mu_r-1}\binom{\mu_r-1}{\nu_r-1}_q\binom{\mu_{r-1}-\nu_r+1}{\nu_{r-1}-\nu_r+1}_q\cdot q^{(\nu_r-1)(\mu_{r-1}-\nu_{r-1}-1)}P
\end{align*}
And
\begin{align*}
		V(\mu',\nu')+(q^{\mu_r-1}-1)V(\mu''',\nu')-q^{\mu_r-1}V(\mu'',\nu')=\\
		=q^{\mu_r-1}\binom{\mu_r-1}{\nu_r-1}_q\cdot q^{(\nu_r-1)(\mu_{r-1}-\nu_{r-1}-1)}\left(\binom{\mu_{r-1}-\nu_r+1}{\nu_{r-1}-\nu_r+1}_q-\binom{\mu_{r-1}-\nu_r}{\nu_{r-1}-\nu_r+1}_q\right)P\\
		=q^{\mu_r-\nu_r}\binom{\mu_r-1}{\nu_r-1}_q\binom{\mu_{r-1}-\nu_r}{\nu_{r-1}-\nu_r}_q\cdot q^{\nu_r(\mu_{r-1}-\nu_{r-1})}P
	\end{align*}
	Thus
	\begin{align*}
		V(\mu',\nu)+V(\mu',\nu')+(q^{\mu_r-1}-1)V(\mu''',\nu')-q^{\mu_r-1}V(\mu'',\nu')=\\
		\left(\binom{\mu_r-1}{\nu_r}_q+q^{\mu_r-\nu_r}\binom{\mu_r-1}{\nu_r-1}_q\right)\binom{\mu_{r-1}-\nu_r}{\nu_{r-1}-\nu_r}_q\cdot q^{\nu_r(\mu_{r-1}-\nu_{r-1})}P=V(\mu,\nu)
	\end{align*}
\end{proof}

Before we can turn to the proof of Theorem \ref{thmJblock} we need a result from 
 \cite{Erdelyi-Toth}. There in Theorem 18.1. the Jordan block structure of \(a\), i.e. the dual partition $\lambda=\mu'$ was used to describe the recursion algorithm. To state this theorem we denote \(K(a)\) by \(K_\lambda(\alpha)\), and will drop \(\alpha\) from the notation since  it is fixed for us. In particular \(K_{[1]}=K_{[1]}(\alpha)=K(\alpha)\) is the classical Kloosterman. 
 We also let \(K_{[\;]}=1\).
   
For example with this notation Theorem 1.4 of \cite{Erdelyi-Toth} says that for \(\alpha\neq 0\)
\begin{equation}\label{eq:recursion-scalar}
K_{[1^n]}=q^{n-1}K_{[1]} K_{[1^{n-1}]} + q^{2n-2}(q^{n-1}-1)\, K_{[1^{n-2}]}.
\end{equation}
In general the result in \cite{Erdelyi-Toth} gives for  \(\lambda=[n_l^{k_l},n_{l-1}^{k_{l-1}},...,n_1^{k_1}] \), with \(n_l>n_{l-1}>...>n_1\) that 
\begin{equation}\label{eq:recursion-thm-old-version}
K_\lambda =
q^{n-1} K_{[1]}K_{\lambda'} - 
q^{2n-2}K_{\lambda''} -
(q^{k_l-1}-1) q^{2n-2} \left(K_{\lambda''}-K_{\lambda'''} \right),
\end{equation}

where \(\lambda'=[n_{l}^{k_l-1},n_l-1,n_{l-1}^{k_{l-1}},...,n_1^{k_1} ] \), \(\lambda''=[n_{l}^{k_l-1},n_l-2,n_{l-1}^{k_{l-1}},...,n_1^{k_1} ]\) and \(\lambda'''= [n_{l}^{k_l-2},(n_l-1)^2,n_{l-1}^{k_{l-1}},...,n_1^{k_1} ]\) reordered into a monotonic sequence, if needed.

Clearly it is more convenient to use $\mu$ as in (\ref{eq:mu-def}), and we reformulate this recursion formula for the matrix Kloosterman sums in terms of it. To denote this shift we will write $K^\mu=K_\lambda$, whenever \(\lambda=\mu'\). Then we have

\begin{theorem}
Let $\mu=[\mu_1\geq\mu_2\geq\dots\geq\mu_r]\vdash n$ be a partition. Then
\[K^{\mu}=q^{n-1}K^{[1]}K^{\mu'}-q^{2n+\mu_r-3}K^{\mu''}+(q^{\mu_r-1}-1)q^{2n-2}K^{\mu'''}.\]
Here \(K^{[\;]}=1\) and  $\mu'=[\mu_1\geq\mu_2\geq\dots\geq\mu_{r-1}\geq\mu_r-1]$, $\mu''=[\mu_1\geq\mu_2\geq\dots\geq\mu_{r-1}-1\geq\mu_r-1]$ and $\mu'''=[\mu_1\geq\mu_2\geq\dots\geq\mu_{r-1}\geq\mu_r-2]$ are as in (\ref{eq:mu-primes}). 
\end{theorem}
%

Note that with our interpretation those terms which correspond to a partition with a negative entry are omitted from the sum.  With this understanding there is no need to split into cases, neither to reorder the elements of the partitions to get a decreasing sequence.

\begin{proof}
Let \(\lambda\) be a partition with \(\mu=[\mu_1,...,\mu_r]\) be its dual. Let \(r_{i}=r_i(\lambda)\) be the number of parts of \(\lambda\) which are equal to \(i \geq 1\). Then \(\mu_{i}=\sum_{j \geq i} r_{j}.\) 

The proof of the reformulation is an elementary argument unfolding this relation, the details of which are omitted.
\end{proof}

\begin{proof}[Proof of Therem \ref{thmJblock}]
We want to show that
\[K^\mu=(-1)^nq^{n(n-1)/2}\sum_{k=0}^n\lambda^k\bar\lambda^{n-k}\sum_{\nu\vdash k}V(\mu,\nu).\]
For convenience we treat the sum so that the terms $V(\mu,\nu)=0$ corresponding to partitions $\nu$ such that $\nu_j>\mu_j$ for some $j$ are included. 

We have developed a  recurrence relation for both sides, it is enough to check that they are compatible.

Since $K^{[\;]}=1=V([~],[~])$ and \[K^{[1]}=-(\lambda+\bar\lambda)=-(V([1],[1])\lambda+V([1],[~])\bar\lambda),\]
the initial conditions match.

Proceeding inductively we get 
\begin{eqnarray*}
q^{n-1}K^{[1]}K^{\mu'}=(-1)^nq^{n(n-1)/2}(\lambda+\bar\lambda)\sum_{k=0}^{n-1}\lambda^k\bar\lambda^{n-k-1}\sum_{\nu\vdash k}V(\mu',\nu)\\
q^{2n+\mu_r-3}K^{\mu''}=(-1)^nq^{n(n-1)/2}\lambda\bar\lambda q^{\mu_r}\sum_{k=0}^{n-2}\lambda^k\bar\lambda^{n-k-2}\sum_{\nu\vdash k}V(\mu'',\nu')\\
q^{2n-2}(q^{\mu_r-1}-1)K^{\mu'''}=(-1)^nq^{n(n-1)/2}\lambda\bar\lambda (q^{\mu_r}-1)\sum_{k=0}^{n-2}\lambda^k\bar\lambda^{n-k-2}\sum_{\nu\vdash k}V(\mu''',\nu').
\end{eqnarray*}

After simplifying, we compare the coefficient of $\lambda^k\bar\lambda^{n-k}$. On the right-hand side we have
\begin{equation}\label{eq:RHS-coeff}
\sum_{\nu\vdash k}V(\mu',\nu)+\sum_{\nu'\vdash k-1}\left(V(\mu',\nu')-q^{\mu_r-1}V(\mu'',\nu')+(q^{\mu_r-1}-1)V(\mu''',\nu')\right),
\end{equation}
where we wrote many extra terms $V(\mu',\nu)=0$ for $\nu\vdash n$ and $V(\mu'',\nu')=V(\mu''',\nu')=0$ where $\nu'\vdash n-1$.

We want to show that the sum in (\ref{eq:RHS-coeff})  produces $\displaystyle{\sum_{\nu\vdash k}V(\mu,\nu)}$. 

Fix $\nu\vdash k$ and $\nu'=[\nu_1\geq\nu_2\geq\dots\geq\nu_{r-1}\geq \nu_r-1]\vdash k-1$. Note that by Lemma \ref{q-binom-eqn} we are almost done. To prove the theorem, we only need to treat 
those $\nu'$-s do not arise as $[\nu_1\geq\nu_2\geq\dots\geq\nu_{r-1}\geq\nu_r-1]$ for which $\nu_{r-1}=\nu_r$. Separate the following cases to show that the contribution of the terms including $\nu'$ is 0:

\begin{itemize}
\item[\textbf{Case 1.}] $\nu'_{r-1}=\nu'_r<\mu_r-1$. Then $V(\mu',\nu')=V(\mu'',\nu')=V(\mu''',\nu')$, so \[V(\mu',\nu')-q^{\mu_r-1}V(\mu'',\nu')+(q^{\mu_r-1}-1)V(\mu''',\nu')=0.\]
\item[\textbf{Case 2.}] $\nu'_{r-1}=\nu'_r>\mu_r-1$. Then $V(\mu',\nu')=V(\mu'',\nu')=V(\mu''',\nu')=0$.
\item[\textbf{Case 3.}] $\nu'_{r-1}=\nu'_r=\mu_r-1$. Then $V(\mu,\nu''')=0$ and the product formula for $V(\mu',\nu')$ and $V(\mu'',\nu')$ differ only in the term $j=r-1$ and there we have
\begin{eqnarray*}
\frac{V(\mu',\nu')-q^{\mu_r-1}V(\mu'',\nu')}{\displaystyle{\prod_{j\neq r-1}\binom{\mu_j-\nu_{j+1}}{\nu_j-\nu_{j+1}}_q\cdot q^{\nu_{j+1}(\mu_j-\nu_j)}}}=\\
\binom{\mu_{r-1}-\mu_r+1}{0}_q\cdot q^{(\mu_r-1)(\mu_{r-1}-\mu_r+1)}-q^{\mu_r-1}\binom{\mu_{r-1}-\mu_r}{0}_q\cdot q^{(\mu_r-1)(\mu_{r-1}-\mu_r)}
\end{eqnarray*}
which is also 0.
\end{itemize}
\end{proof}

\begin{corollary}
$K^\mu=O_q(q^{n^2/2+d})$, where $\displaystyle{d=\frac14\sum_{j=1}^r(\mu_j^2-\delta(\mu_j))}$, where $\delta(m)=1$ if $m$ is odd and $\delta(m)=0$ if $m$ is even.
\end{corollary}

For this it is enough to note that
\[\binom{\mu_j-\nu_{j+1}}{\nu_j-\nu_{j+1}}_q\cdot q^{\nu_{j+1}(\mu_j-\nu_j)}=O_q(q^{\nu_j(\mu_j-\nu_j)}).\]
Thus to maximize the order of $V(\mu,\nu)$ the optimal choice for $\nu_j$ is the closest integer to $\mu_j/2$. 

\begin{corollary}\label{corweights}
If $\mu_1>1$, then $K^\mu$ must admit Frobenius eigenvalues of two different weights.
\end{corollary}

This follows from the fact that $K^\mu$ is a polynomial of $q$, $\lambda$ and $\bar\lambda$ and by the previous corollary we have that this polynomial has a monomial of weight more than $n^2$ and also one with weight $n^2$ (the one corresponding to the trivial irreducible subspaces).

\vspace{1em}

Now we can prove what happens for non-regular matrices:

\begin{proof}[Proof of Theorem \ref{thmreg}]
Assume $a$ is not regular. Observe that $H^*(a)$ is the tensor product of the cohomology complexes corresponding to the matrix $j_l$ containing the Jordan blocks for each eigenvalue $\lambda_l$ of $a$ (for $1\leq l\leq r$) tensored with a power of the trivial sum (directly follows from \cite{Erdelyi-Toth}, Proposition 13.1.):
\[H^*(a)=\left(\bigotimes_{l=1}^rH^*(j_l)\right)\otimes\left(\bigotimes_{m=1}^dH^*(\mathbb{A}^1,0)\right),\]
with $d=\sum_{1\leq l<l'\leq r}ll'$.

Assume $a$ has a nonzero eigenvalue $\alpha$ such that the eigenspace is not one-dimensional (that means $\mu_1>1$ in the above notation). As the term corresponding to $\alpha$ has two different weights (Corollary \ref{corweights}), so does $H^*(a)$, hence $K(a)$ is not concentrated to a single degree by Theorem \ref{thmpurity}.
\end{proof}

\end{document}